\documentclass[12pt]{amsart}
\usepackage{amssymb,latexsym}
\newtheorem{theorem}{Theorem}
\newtheorem{lemma}[theorem]{Lemma}

\newtheorem{corollary}[theorem]{Corollary}

\newtheorem{proposition}[theorem]{Proposition}

\def\se{{\, \subseteq \, }}

\title[cohomology of arithmetic groups]{An infinitely generated virtual cohomology group for noncocompact arithmetic groups over function fields}

\thanks{The author gratefully acknowledges the support of the National Science Foundation.}

\author{Kevin Wortman}

\begin{document}

\begin{abstract} 
Let ${\bf G}(\mathcal{O}_S)$ be a noncocompact irreducible arithmetic group over a global function field $K$ of characteristic $p$, and let $\Gamma$ be a finite-index, residually $p$-finite subgroup of ${\bf G}(\mathcal{O}_S)$.
We show that the cohomology of $\Gamma$ in the dimension of its associated Euclidean building with coefficients in the field of $p$ elements is infinite.
\end{abstract}

\maketitle

Let $K$ be a global function field that contains the field with $p$ elements, $\mathbb{F}_p$. We let $S$ be a finite nonempty set of inequivalent valuations of $K$.
The ring $\mathcal{O}_S \subseteq K$ will denote  the corresponding ring of $S$-integers.
For any $v \in S$, we let $K_v$ be the completion of $K$ with respect to $v$ so that $K_v$ is a locally compact field. 

We denote by $\bf G$ a connected noncommutative absolutely almost simple $K$-group, and we let
$$k({\bf G},S)=\sum_{v\in S}\text{rank}_{K_v}\bf G$$ so that $k({\bf G},S)$ is the dimension of the Euclidean building on which the arithmetic group ${\bf G}(\mathcal{O}_S)$ acts as a lattice. Thus for example, $k({\bf SL_n},S)=|S|(n-1)$.

If $\bf G$ is $K$-anisotropic, then ${\bf G}(\mathcal{O}_S)$ contains a torsion-free finite-index subgroup that acts freely and cocompactly on a Euclidean building of dimension $k({\bf G},S)$. 
Determining the finiteness properties of arithmetic groups ${\bf G}(\mathcal{O}_S)$ in the case that $\bf G$ is $K$-isotropic has been more difficult. The model for the $K$-isotropic case was provided by the following theorem of Stuhler \cite{St}.

\begin{theorem}\label{t:st}
The arithmetic group ${\bf SL_2}(\mathcal{O}_S)$ is of type $F_{k({\bf SL_2},S)-1}$, and if $\Gamma$ is  any finite-index subgroup of ${\bf SL_2}(\mathcal{O}_S)$ whose only torsion elements are $p$-elements, then
${\rm H}^{ k({\bf SL_2},S)}(\, \Gamma \,;\, \mathbb{F}_p\,)$ is infinite.
\end{theorem}

Recall that a group $\pi$ is \emph{of type }$F_n$ if there exists a $K(\pi,1)$ with finite $n$-skeleton. 

It is well-known, by Selberg's Lemma, that ${\bf SL_2}(\mathcal{O}_S)$, or that any arithmetic group over function fields ${\bf G}(\mathcal{O}_S)$ as above, contains a finite-index subgroup whose only torsion elements are $p$-elements.

Bux-K\"{o}hl-Witzel \cite{B-G-W} completely generalized ``half" of Theorem~\ref{t:st} with the following theorem.

\begin{theorem}\label{t:bgw}
If $\bf G$ is $K$-isotropic, then ${\bf G}(\mathcal{O}_S)$ is of type $F_{k({\bf G},S)-1}$.
\end{theorem}

Important evidence for the theorem of Bux-K\"{o}hl-Witzel was contributed by Behr \cite{Behr}, Abels \cite{Abels}, Abramenko \cite{Abramenko}, and Bux-Wortman \cite{B-W2}.

There are now three proofs that ${\bf G}(\mathcal{O}_S)$ as in Theorem~\ref{t:bgw} is not of type $F_{k({\bf G},S)}$ due to Bux-Wortman \cite{B-W1}, Bux-K\"{o}hl-Witzel \cite{B-G-W}, and Kropholler \cite{K} as observed by  Gandini \cite{Ga}. However, outside of the case that $k({\bf G},S)=1$, the ``second half" of Stuhler's Theorem~\ref{t:st} had not been generalized to include any other arithmetic groups. This paper uses the results of Bux-K\"{o}hl-Witzel and Schulz \cite{Sc} to further generalize the results of Stuhler by proving

\begin{theorem}\label{t:mt}
Suppose $\bf G$ is $K$-isotropic. If $\Gamma$ is a finite-index subgroup 
of ${\bf G}(\mathcal{O}_S)$ that is residually $p$-finite, then ${\rm H}^{k({\bf G},S)}(\, \Gamma \,;\, \mathbb{F}_p\,)$ is infinite.
\end{theorem}

A group $\Gamma$ is  {\emph{residually $p$-finite}} if for any nontrivial $\gamma \in \Gamma$, there is a homomorphism of $\Gamma$ onto a finite $p$-group that evaluates $\gamma$ nontrivially. Such finite-index subgroups of ${\bf G}(\mathcal{O_S})$ are well-known to exist, by Platanov's Theorem, and we provide a proof of their existence in Section~\ref{s:frpf} for completeness.

To compare Theorems~\ref{t:st} and~\ref{t:mt}, notice that any torsion element of a residually $p$-finite group has order a power of $p$. The author does not know of an example of a finite-index subgroup $\Gamma \leq {\bf G}(\mathcal{O}_S)$ whose only torsion elements are $p$-elements, but such that $\Gamma$ is not residually $p$-finite.

As an example of Theorem~\ref{t:mt}, there is a finite-index subgroup of ${\bf SL_n}(\mathcal{O}_S)$ whose cohomology in dimension $|S|(n-1)$ with coefficients in $\mathbb{F}_p$ is infinite. In particular, there is a finite-index subgroup $\Gamma$ of ${\bf SL_n}(\mathbb{F}_p[t])$ such that $H^{n-1}(\Gamma;\mathbb{F}_p)$ is infinite.

\subsection{Outline of the proof.} 
To prove Theorem~\ref{t:st}, Stuhler analyzed the cell stabilizers of the ${\bf SL_2}(\mathcal{O}_S)$-action on the associated Euclidean building which is a product of regular $(p+1)$-valent trees. The cell stabilizers of  $\Gamma$ as in Theorem~\ref{t:st}  are products of the group $\mathbb{F}_p$, but  the cell stabilizers of a random arithmetic group acting on its associated Euclidean building are more difficult to describe and to work with, so our proof of Theorem~\ref{t:mt} proceeds in a different direction.

The main tool in our proof of Theorem~\ref{t:mt} is the work of Bux-K\"{o}hl-Witzel, and we spend a good portion of the beginning of our proof recalling their work. Let $k=k({\bf G},S)$ and let $X$ be the Euclidean building that 
${\bf G}(\mathcal{O}_S)$ acts on as a lattice. Bux-K\"{o}hl-Witzel  finds a ${\bf G}(\mathcal{O}_S)$-invariant, cocompact, $(k-2) $-connected complex $X_{k-2} \se X$. We attach $k$-cells and $(k+1)$-cells to $X_{k-2}$ to produce a $k$-connected complex $X_k$ endowed with a $\Gamma$-action and a $\Gamma$-equivariant map $\psi : X_k \rightarrow X$.

We find an unbounded sequence of points $\Gamma y_n \in \Gamma \backslash X$, and a sequence of normal subgroups $\Gamma_n$ of $\Gamma$ with index a power of $p$ such that each $y_n \in X$ is contained in  a neighborhood of $X$ that injects into $ \Gamma_n \backslash X $, and such that the $p$-group $\Gamma / \Gamma _n $ acts on the homology of the image of the neighborhood in the quotient, with coefficients in $\mathbb{F}_p$.  The action of the $p$-group on the homology group produces a functional that nontrivially, and $\Gamma$-invariantly, evaluates the image under $\psi$ of the attached $k$-cells in $X_k$. Therefore, for each $n$, we have an assignment of $k$-cells in $\Gamma \backslash X_k$ to elements of $\mathbb{F}_p$. This produces an infinite sequence in  ${\rm H}^{k}(\, \Gamma \backslash X_k  \,;\, \mathbb{F}_p\,)$.
The group $\Gamma$ may not act freely on $X_k$, but the lack of freeness is confined to a cocompact subspace of $X_k$, namely $X_{k-2}$, and that implies that  ${\rm H}^{k}(\, \Gamma \,;\, \mathbb{F}_p\,)$ is infinite.

\subsection{Acknowledgements} Thanks to Dan Margalit for a discussion that lead to the proof of this result and to Kai-Uwe Bux for careful explanations of his work with K\"{o}hl and Witzel.

Thanks also to Mladen Bestvina, Morgan Cesa, Brendan Kelly, Amir Mohammadi, and Dave Morris for helpful conversations, and to Stefan Witzel for pointing out some improvements that were made to an earlier version of this paper.

\section{Preliminaries on ${\bf G}(\mathcal{O}_S)$ and its action on a Euclidean building}

This section establishes some conventions for notation.

\subsection{Basic group structure.}\label{s:simpleroots} Let $K$, $\mathcal{O}_S$, and $\bf G$  be as in Theorem~\ref{t:mt}. Because $\bf G$ is $K$-isotropic, it contains a proper minimal $K$-parabolic subgroup $\bf{J}$.
 Let  $\bf A$ be a maximal $K$-split torus in $\bf J$, and let  $\bf P$ be a maximal proper $K$-parabolic subgroup of $\bf G$ that contains $\bf J$.

Recall the Langlands decomposition that $${\bf P}={\bf U}{\bf H}{\bf T} $$
where $\bf U$ is the unipotent radical of ${\bf P}$, $\bf H$  is a reductive $K$-group with $K$-anisotropic center, $\bf T$ is a 1-dimensional connected subtorus of $\bf A$, and $\bf T$ commutes with $\bf H$.

In the remainder of this paper we denote the product over $S$ of local points of a $K$-group by ``unbolding", so that, for example, $$G=\prod _{v\in S}{\bf G}(K_v)$$

\subsection{Euclidean building} Let $X$ be the Euclidean building for the semisimple group $G$. We let $k=k({\bf G},S)$ so that  $k={\text{dim}}(X)$.

For each $v \in S$ we choose a maximal $K_v$-split torus in $\bf J$ that contains $\bf A$, and name it ${\bf A}_v$. We let $\Sigma \se X$ be the apartment corresponding to the group $\prod_{v \in S} {\bf A}_v(K_v)$.

\section{Review of Bux-K\"{o}hl-Witzel and an unbounded sequence of points $y_n \in X$}\label{s:BGW}

Our proof makes use of two results from Bux-K\"{o}hl-Witzel \cite{B-G-W}: the existence of a ${\bf G}(\mathcal{O}_S)$-invariant, $(k-2)$-connected subcomplex $X_{k-2} \se X$ that is cocompact modulo ${\bf G}(\mathcal{O}_S)$, and a lemma that will allow us to extend certain ``local" $k$-disks about neighborhoods of points in $X$ to ``global" $k$-disks in $X$---Lemma~\ref{l:extendstar} and Corollary~\ref{c:extend} below. Most of this section is devoted to recalling the work of Bux-K\"{o}hl-Witzel. For details omitted from the account in this paper, see \cite{B-G-W}.

We will use the notation of \cite{B-G-W} in our Section~\ref{s:BGW} except for the following: we will refer to cells in the spherical building for $G$ by the parabolic groups they represent. For example, if $g \in G$ and we write that $g \in P$, then we are treating $P$ as a parabolic group, but if $x$ is a point in the visual boundary of $X$ and we write that $x \in P$, then we are treating $P$ as the simplex in the visual boundary of $X$ that corresponds to $P$. The correct interpretation should always be clear from context.

\subsection{Busemann function for $P$.}\label{s:bf}
For each $v \in S$, let $X_v$ be the Euclidean building for ${\bf G}(K_v)$, so that $X=\prod_{v\in S}X_v$. If $\mathcal{O}_v \se K_v$ is the ring of integers, then we let $x_v$ be the vertex in $X_v$ stabilized by ${\bf G}(\mathcal{O}_v)$.

Let $\mathbb{A}_K$ be the ring of adeles for $K$, and let $\mathbb{A}_S$ be the subring of $S$-adeles. The group  ${\bf G}(\mathbb{A}_S)$ has a natural left action on $X$. 
Given a point $y \in X$ we let $ {\bf G}(\mathbb{A}_S)_y$ be the stabilizer of $y$ in  ${\bf G}(\mathbb{A}_S)$.

Following Harder (\cite{H}) and \cite{B-G-W}, for any $y \in \prod_{v \in S}{\bf G}(K_v)x_v$ we let $$\tilde{\beta}_{P}(y)=\log_q \Big[{\text{vol}}\big[{\bf U}(\mathbb{A}_K) \cap {\bf G}(\mathbb{A}_S)_y \big]\Big]$$
where $q$ is the cardinality of the field of constants in $K$.

We let $\chi _{{\bf P}}$ be the canonical character of ${\bf P}$. (See Section 1.3 \cite{H} for the definition of $\chi _{{\bf P}}$.) The essential feature of $\chi _{{\bf P}}$ that will be used below is that the determinant of conjugation by $g \in P$ on $U$ is $\chi _{{\bf P}} (g)$.

If $g \in P$, then we have the following transformation rule from Harder \cite{H} Satz 1.3.2: 
$$ \tilde{\beta}_{P}(g y)=   \tilde{\beta}_{P}( y) + \log _q ( || \chi _{{\bf P}} (g)||) $$  where $||\cdot||$ denotes the idele norm.
 (There is a difference in sign in the line above with \cite{H} and  \cite{B-G-W} that comes from our convention of using left actions in this paper rather than right actions as in  \cite{H} and  \cite{B-G-W}.)

Recall that a Busemann function on the Euclidean building $X$ is given by first choosing a unit speed geodesic $\rho \se X$ and then assigning to any point $x\in X$ the limit as $t \to \infty$ of the difference between the distance between $\rho (t)$ and $\rho(0)$ and the distance between $\rho(t)$ and $x$.

\begin{proposition}\label{p:factores}
There is some $s>0$ and a Busemann function $\beta _{P}: X \rightarrow \mathbb{R}$ such that $\beta _{P}(y)=\tilde{\beta}_{P}(y)$ for all  $y \in \prod _{v \in S}{\bf G}(K_v) x_v$, and such that $\beta _{P}$ is nonconstant on factors of $X$.
\end{proposition}

\begin{proof}
This is Proposition 12.2 of  \cite{B-G-W}.
\end{proof}

\begin{lemma}\label{l:inv}
The Busemann function $\beta_{P}$ is invariant under the actions of $U$, $H$, and ${\bf T}(\mathcal{O}_S)$ on $X$, and thus is invariant under the action of ${\bf P}(\mathcal{O}_S) \leq UH{\bf T}(\mathcal{O}_S)$.
\end{lemma}

\begin{proof}
Any $K$-defined character on ${\bf P}$, including the canonical character $\chi _{{\bf P}} $, evaluates ${\bf U}$ trivially since it is unipotent and ${\bf H}$ trivially since it is reductive with $K$-anisotropic center. Thus the result for $U$ and $H$ follows from the transformation rule above. 

Similarly, we need to observe that $||\chi_{{\bf P}}(t)||=1$ for any
 $t \in {\bf T}(\mathcal{O}_S)$. This follows from the product formula (since $\chi_{{\bf P}}(t) \in K$) and from the fact that ${\bf T}(K_w)$ is bounded if $w \notin S$.
 \end{proof}

\subsection{Descending chambers at a vertex}
Given a vertex $x \in X$, we let ${\rm St}(x) \se X$ denote the star of $x$, the union of all chambers in $X$ that contain $x$. Thus, the boundary of the star -- denoted as $\partial {\rm St}(x)$ -- is the link of $x$.

We let ${\rm St}^\downarrow(x)$ denote the union of chambers $\mathfrak{C} \se X$ containing $x$ with the property that $\beta_{P}(z)<\beta_{P}(x)$ for all $z \in \mathfrak{C}$ with $z \neq x$. We let $B{\rm St}^\downarrow(x)={\rm St}^\downarrow(x) \cap \partial {\rm St}(x)$.

Recall that a special vertex  $x \in \Sigma $ is a vertex that is contained in a representative from each parallel family of walls in the Coxeter complex $\Sigma $. Thus, the Coxeter complex of an apartment in the spherical building $\partial {\rm St}(x)$ is isomorphic to the Coxeter complex of an apartment in the boundary of $X$ when $x$ is special.

The following result is due to Schulz \cite{Sc}.

\begin{lemma}\label{l:Schulz}
If $x \in X$ is  a special vertex, then $B {\rm St}^\downarrow(x)$ is homotopy equivalent to a noncontractible wedge of $(k-1)$-spheres.
\end{lemma}

\begin{proof}
Recall that the Busemann function $\beta_{P}$ is nonconstant on the factors of $X$. Since $x$ is a special vertex, the join factors of $\partial {\rm St}(x)$ correspond to the factors of $X$. Therefore, $\beta_{P}$ is nonconstant on the join factors of $\partial {\rm St}(x)$. That is to say, in the terminology used in \cite{B-G-W}, the ``vertical part" of  $\partial {\rm St}(x)$ is $\partial {\rm St}(x)$ in its entirety.

Notice that $B {\rm St}^\downarrow(x)$ is exactly the maximal subcomplex of $\partial {\rm St}(x)$ that is supported on the complement of the closed ball of radius $\frac \pi 2$ around the gradient direction of $\beta_{P}$ in $ \partial{\rm St}(x)$. Thus, by Theorem B of \cite{Sc} -- restated in Theorem 4.6 of \cite{B-G-W} -- $B {\rm St}^\downarrow(x)$ is $(k-1)$-dimensional, $(k-2)$-connected, and noncontractible.
\end{proof}

See also Theorem A.2 of Dymara-Osajda \cite{D}.

\subsection{Reduction datum}\label{s:rd} If $\bf M_a$ is a maximal proper $K$-parabolic subgroup of $\bf G$, then we can define a Busemann function $\beta_{M_a}$ with respect to $\bf M_a$ similarly to how we defined $\beta_{P}$ with respect to ${\bf P}$.

In \cite{B-G-W}, and following \cite{H}, there are real constants $r<R$ such that the collection of Busemann functions $\beta_{M_a}$ forms what is called  a \emph{uniform ${\bf G}(\mathcal{O}_S)$-invariant and cocompact reduction datum}. (See Theorem 1.9 of \cite{B-G-W}.) The remainder of Section~\ref{s:rd} is a recollection of what this sort of datum entails. In Section~\ref{s:rd} we will use $\bf M_a$ to denote a maximal proper $K$-parabolic subgroup of $\bf G$. We will use $\bf M_i$ to denote a minimal $K$-parabolic subgroup of $\bf G$.

For $x \in X$ and a $K$-parabolic subgroup ${\bf Q} \leq {\bf G}$, we let $\beta _Q(x)$ be the maximum of all $\beta _ {M_a}(x)$ with ${\bf Q} \leq {\bf M_a}$.

Given an apartment $\Sigma' \se X$ that contains $Q$ as a cell in its boundary, and given $t \in \mathbb{R}$, we let $$Y_{\Sigma', {Q}}(t)=\{\,x \in \Sigma' \mid \beta_{ Q}(x)\leq t \,\}$$ This set is convex in $\Sigma '$ as it is the intersection of the convex sets $\Sigma ' \cap \beta_{M_a}^{-1}(\mathbb{R}_{\leq t})$ for $\bf M_a$ containing $\bf Q$. Thus, there is a closest point projection $${\text{pr}}^t_{\Sigma', {Q}}: \Sigma '\rightarrow Y_{\Sigma', {Q}}(t)$$

The group $\sigma_t(x , { Q})$ is defined to be the group of $\prod_{v\in S}K_v$-points of the intersection of all ${\bf M_a}$ that  contain $\bf Q$ and such that $\beta_{ M_a}({\text{pr}}^t_{\Sigma', { Q}}(x))=t$. We have that $\sigma_t(x , { Q}) \geq Q$ (as groups, not as cells in the boundary) and we say that $Q$ $t$-\emph{reduces} $x\in X$ if $\sigma _t(x,  { Q})={Q}$.

To say that the collection of $\beta _{M_a}$ is an $(r,R)$ \emph{reduction datum} for $r<R$ means that if $\bf M_i$ is a minimal $K$-parabolic subgroup of $\bf G$ that $r$-reduces $x \in X$, then  ${\bf M_i} \leq \sigma _R (x, M_i) $.

To say that the reduction datum is \emph{uniform} means that there exists a constant $d$ such that any point in a subset of $X$ whose diameter is less than $d$ can be $r$-reduced by a common minimal $K$-parabolic. We can assume, as in \cite{B-G-W}, by perhaps choosing a lesser $r$, that $d$ is greater than the diameter of closed stars of cells in $X$.

The reduction datum is ${\bf G}(\mathcal{O}_S)$-\emph{invariant} since $$\beta_{^\gamma M_a}(\gamma x)=\beta_{  M_a}(x)$$ for all $x \in X$, $\gamma \in {\bf G}(\mathcal{O}_S)$, and maximal proper $K$-parabolic $\bf M_a$. (Here $^\gamma M_a = \gamma M_a \gamma ^{-1}$.)

That the reduction datum is \emph{cocompact} means that for any real number $t\geq R$, the set of $x \in X$ for which $\beta _{M_i}(x) \leq t$ for all minimal $K$-parabolics $\bf M_i$ that $r$-reduce $x$ is cocompact with respect to the action of ${\bf G}(\mathcal{O}_S)$.

\subsection{Definition of height}\label{s:h} In \cite{B-G-W}, the reduction datum is used to define a height function $h: X \rightarrow \mathbb{R}_{\geq 0}$. In Section~\ref{s:h}, we recall this definition.

Choose a special vertex $z \in \Sigma$, and let $W_z$ be the spherical Coxeter group that fixes $z$ in $\Sigma $. 

The affine space $\Sigma$ may be realized as a vector space with origin $z$.
Let $V_z$ be the set of all differences of vertices in $\Sigma$ whose closed stars intersect, where we regard vertices in this context as vectors in $\Sigma$. Notice that $V_z$ is finite.

We let $D=W_zV_z$.
Again, realizing points of $D$ as vectors of the vector space $\Sigma $ with origin $z$, we let 
$$Z(D)=\Big\{\, \sum_{d \in D}a_dd \mid 0 \leq a_d \leq 1 \text{ for all } d \in D \Big\}$$

The set $Z(D) \se \Sigma $ depended on the choice of vertex $z$, but modulo isometric translations  of $\Sigma $, $Z(D)$ is defined intrinsically in terms of the geometry of $\Sigma$. Furthermore, if $\Sigma ' \se X$ is any apartment in $X$, then $\Sigma '$ is isometric to $\Sigma $ as Coxeter complexes, and thus $x+Z(D)$ is a well-defined subset of $\Sigma ' $ for any $x \in \Sigma '$.

To define a height function, a suitably large $R^*>R$ is chosen. For any apartment $\Sigma ' \se X$, any  $x \in \Sigma '$, and any minimal $K$-parabolic ${\bf M_i}$ such that $M_i$ represents a cell in the boundary of $\Sigma '$ that $r$-reduces $x$; the point $x^*_{\Sigma', {M_i}}$ is defined to be the closest point to $x$ in $Y_{\Sigma ', {M_i}}(R^*)-Z(D)$. Then $h(x)$ is defined as the distance between $x$ and $x^*_{\Sigma', {M_i}}$, and it is shown in Proposition 5.2 of \cite{B-G-W} to be independent of $\Sigma'$ or $\bf M_i$.

If $h(x)>0$, then $e(x)$ is defined as the point in the visual boundary of $\Sigma'$ that is determined as the limit point of the geodesic ray in $\Sigma ' $ from $x^*_{\Sigma', {M_i}}$ through $x$.  The point $e(x)$ is also shown to be independent of $\Sigma'$ or $\bf M_i$ in Proposition 5.2 of \cite{B-G-W}.
 If we let $\sigma (x)$ denote the group of $\prod_{v\in S}K_v$-points of the $K$-parabolic subgroup of $\bf G$ that is minimal with respect to the property that $\sigma (x)$ contains every $\sigma _R(x,M_i)$ for which $M_i$ $r$-reduces $x$, then $e(x)\in \sigma (x)$.

As the reduction datum used in this section is $\mathbf{G}(\mathcal{O}_S)$-invariant, we have that $h (\gamma x)=h(x)$ for any $\gamma \in \mathbf{G}(\mathcal{O}_S)$. And if $h(x)>0$, then $e(\gamma x)=\gamma e(x)$ and $\sigma (\gamma x)= \,^\gamma \sigma (x)$.

The subsets of $X$ whose values under $h$ are bounded from above are shown to have bounded quotient on ${\bf G} (\mathcal{O}_S) \backslash X$ (See Proposition 2.4 and Observation 5.5 of \cite{B-G-W}).

\subsection{Choice of $y_n$.}
We still have more to discuss about the results of \cite{B-G-W}, but we take a short break from our account of \cite{B-G-W} to establish a sequence of points in $X$ that will be used throughout our proof in this paper.

\begin{lemma}\label{l:yn} Let $N^*>0$ be twice the  maximum diameter of stars in $X$. We can choose $R^*\gg 0$ as above to satisfy the following:
There is a constant $C^* \in \mathbb{R}$, and  a geodesic ray $\ell_Y \se \Sigma$ that limits to a point $\ell_Y(\infty)$ in the simplex $P$ and is orthogonal to level sets of $\beta _P$ in $\Sigma$, such that every point $z$ in the $N^*$-neighborhood of $U\ell _Y$ in $X$ is $r$-reduced by $J$, has $h(z)=\beta_P(z)+C^*>0$, and has $e(z) = \ell_Y(\infty) \in P$.

Furthermore, there is a sequence of special vertices $y_n \in \Sigma$ that are contained in chambers of $\Sigma$ that intersect $\ell_Y$, such that $\beta_P(y_n)$ is a strictly increasing sequence of numbers, and such that the set of all $ (y_n)^*_{\Sigma, P}$ is a bounded set.
\end{lemma}

\begin{proof}
There are $\text{rank}_{K}{\bf G}\leq {\text{dim}}(\Sigma)$ maximal proper $K$-parabolic subgroups that contain $\bf J$. The space $Y_{\Sigma , J}(R^*) \se \Sigma$ is the intersection of one half-apartment of $\Sigma$ for every maximal proper $K$-parabolic subgroup that contains $\bf J$, and the set $\beta_{P}^{-1}( R^*)\cap Y_{\Sigma , J}(R^*)$ is an unbounded face of the boundary of $Y_{\Sigma , J}(R^*)$. We call this face $F_{P,R^*}$. It has dimension equal to $\text{dim}(\Sigma)-1$.

We let $$\Omega(r,R^*,J,P)=\{\, x \in \Sigma \mid \sigma _{r} (x ,J)=J \text{ and }   \sigma _{R^*} (x ,J)=P  \,\}$$ 

For $x \in \Sigma$, we let $B_\Sigma(x;N^*) \se \Sigma $ be the ball in $\Sigma$ centered at $x$ with radius $N^*$.
Notice that by replacing $R^*$ with a greater constant, we may assume that there is some $x \in F_{P,R^*} \cap \Omega(r,R^*,J,P)$ such that 
$$  F_{P,R^*} \cap [B_\Sigma(x;N^*) +Z(D)] \se F_{P,R^*} \cap \Omega(r,R^*,J,P)$$ Furthermore, if $y$ is contained in the geodesic ray  $ \ell_Y \se \Sigma$ that begins at $x$, is orthogonal to $ F_{P,R^*}$, and is contained in $ \Omega(r,R^*,J,P)$, then $B_\Sigma(y;N^*) +Z(D) \se \Omega(r,R^*,J,P)$ as long as the distance between $y$ and $x$ is sufficiently large. We replace $\ell _Y$ with a subray so that $B_\Sigma(y;N^*) +Z(D) \se \Omega(r,R^*,J,P)$ for any $y \in \ell_Y$.

If $z$ is contained in the interior of $\Omega(r,R^*,J,P)$, then $e(z)$  is given by the direction of the gradient of $\beta _{P}$ restricted to $\Sigma$ --- which is the direction of $\ell_Y(\infty)$. Thus by Lemma~\ref{l:inv}, $UH e(z)=UH \ell_Y(\infty)=\ell_Y(\infty)=e(z)$. And $T \leq A$ acts trivially on the boundary of $\Sigma$, so we have $P e(z) = UHT e(z)=e(z)$ which implies that
 $e(z) \in P$.

We let $d_0$ be the constant difference of the distance between $y \in \ell_Y$ and  $\Sigma \cap \beta_{P}^{-1}(R^* )$ and the distance between $y +Z(D)$ and $\Sigma \cap \beta_{P}^{-1}( R^* )$. (Note that the latter of the two distances is $h(y)$.) Then for $z \in B_\Sigma(y;N^*)$, $h(z)=\beta_{P}(z)-R^*-d_0$. Thus we let $C^*=-R^*-d_0$.

Again let $y\in \ell_Y $ and now let $z \in B_X(y;N^*)$, where $B_X(y;N^*)$ is the ball in $X$ of radius $N$ that is centered at $y$. We will show that $z$ is $r$-reduced by $J$, has $h(z)=\beta_P(z)+C^*>0$, and has $e(z) = \ell_Y(\infty) \in P$.

For every $v \in S$, let ${\bf J}_v \leq {\bf G}$ be a minimal $K_v$-parabolic subgroup of $\bf G$ such that ${\bf A}_v \leq {\bf J}_v \leq {\bf J}$. We let ${\bf U}_v$ be the unipotent radical of ${ \bf J}_v$, so that  ${\bf U}_v \leq {\bf J}\leq {\bf P}$ and ${\bf U}_v \leq {\bf U}{\bf H}$.

If $X_v$ is the Euclidean building for ${\bf G}(K_v)$, and $\Sigma _v$ is the apartment that ${\bf A}_v(K_v)$ acts on, then because any point in $X_v$ is contained in a ${\bf J}_v(K_v)$ translate of $\Sigma_v$ $$X_v={\bf J}_v(K_v)\Sigma_v={\bf U}_v(K_v) {\bf Z_G( A}_v{\bf )}(K_v)\Sigma_v={\bf U}_v(K_v) \Sigma_v$$ where ${\bf Z_G( A}_v{\bf )}$ is the centralizer of ${\bf A}_v$ in $\bf G$, and thus is a Levi subgroup of ${\bf J}_v$.
Therefore, $$X= \prod_{v \in S}{\bf U}_v(K_v) \Sigma  $$ and there is a distance nonincreasing retraction $$\varrho : X \rightarrow \Sigma$$ defined on each $u \Sigma$ for $u \in \prod_{v \in S}{\bf U}_v(K_v)$ as the map $u^{-1} : u \Sigma \rightarrow \Sigma$.

So for  $z \in B_X(y;N^*)$ we choose $u \in \prod_{v \in S}{\bf U}_v(K_v)$ such that $u^{-1} z \in \Sigma$. Because $\varrho$ is distance nonincreasing and $\varrho(y)=y$, we have that $u^{-1}z \in B_\Sigma(y;N^*)$. By Lemma~\ref{l:inv}  $$\beta _{P}(z)+C^*=\beta _{P}(u^{-1}z)+C^*=h(u^{-1}z)>0$$

If $\bf Q$ is a proper $K$-parabolic subgroup of $\bf G$ containing $\bf J$, then $\bf Q$ contains ${\bf U}_v$ and thus
$u^{-1} {\bf Q}u={\bf Q}$, so applying the clear analogue of Lemma~\ref{l:inv} to each maximal proper $K$-parabolic group containing $\bf J$ yields $u Y_{\Sigma , J}(R^*) =  Y_{u\Sigma , J}(R^*)$ and that $z \in u \Omega(r,R^*,J,P)$ since $u^{-1}z \in B_\Sigma (y;N^*) \se  \Omega(r,R^*,J,P)$.
Thus,  $z$ is $r$-reduced by $uJu^{-1}=J$ and $u^{-1}(z^*_{u \Sigma, J})=(u^{-1}z)^*_{\Sigma, J}$ and $$h(z)=h(u^{-1}z)=\beta _{P}(z)+C^*$$

Furthermore,  as the set $ \Omega(r,R^*,J,P)$ limits to the cell $P$ and $u \in P$, the set $u \Omega(r,R^*,J,P)$ also limits to $P$ and thus
$$e(z)=e(u^{-1}z)=\ell_Y(\infty) \in P$$

To review, we have shown that for any  $z $ in the $N^*$-neighborhood of $\ell_Y$ in $X$ that
 $z$ is $r$-reduced by $J$, has $h(z)=\beta_P(z)+C^*>0$, and has $e(z) = \ell_Y(\infty) \in P$.
We still need to show the same results apply to the weaker condition that $z$ is contained in the  $N^*$-neighborhood of $U\ell_Y$ in $X$. For that, recall that $\bf U$ is unipotent, so ${\bf U}(\mathcal{O}_S)$ is a cocompact lattice in $U$. That is, there is a compact set $B \se U$ such that ${\bf U}(\mathcal{O}_S)B=U$. Since $\ell _Y $ limits to $ P$ and $\bf U$ is the unipotent radical of $\bf P$, any element of $U$ fixes pointwise a subray of $\ell _Y$. Therefore, there is a common subray of $\ell _Y$ that is fixed pointwise by every element of $B$. Thus, by replacing $\ell _Y$ with a subray we may assume that $B$ fixes $\ell _Y$ and thus that $$U \ell _Y = {\bf U}(\mathcal{O}_S) B \ell _Y ={\bf U}(\mathcal{O}_S) \ell _Y$$ Hence, if $z \in U B_X (\ell_Y ;N^*) = {\bf U}(\mathcal{O}_S) B_X (\ell_Y ;N^*)$ then $uz \in B_X (\ell_Y ;N^*)$ for some $u \in {\bf U}(\mathcal{O}_S)$, and since $h$ is ${\bf G}(\mathcal{O}_S)$-invariant and $\beta _P$ is $U$-invariant, $$h(z)=h(uz)=\beta _{P}(uz)+C^*=\beta _{P}(z)+C^*$$ Since the reduction datum is ${\bf G}(\mathcal{O}_S)$-invariant and $uz$ is r-reduced by $J$, we see that $z$ is $r$-reduced by $u^{-1}Ju=J$. Last, since $u \in {\bf U}(\mathcal{O}_S) \leq P$ and $e(uz )\in P$ we have $e(z)=u^{-1}e(uz)=e(uz) =\ell_Y(\infty)$.

To find the sequence of $y_n$, just choose an unbounded sequence of chambers in $\Sigma$ that intersect $\ell_Y$. Any chamber in $X$ contains a special vertex, and this produces the sequence of $y_n$. Because each of the $y_n \in \Sigma$ are a uniformly bounded distance from $\ell_Y$, each $ (y_n)^*_{\Sigma, P} \in F_{P,R^*}$ is a uniformly bounded distance from the point  $ x \in F_{P,R^*}$ .
\end{proof}

In the remainder of this paper, we shall abbreviate ${\rm St}(y_n)$ as $S_n$. Similarly, we shall abbreviate ${\rm St}^\downarrow(y_n)$ and $B{\rm St}^\downarrow(y_n)$ as $S_n^\downarrow$ and $BS^\downarrow_n$ respectively.

\subsection{Morse function}\label{s:morse}
Section~\ref{s:morse} is the final section in which we recount the work of Bux-K\"{o}hl-Witzel. In this section we recall the definition of a combinatorial Morse function from \cite{B-G-W} that is defined on the vertices of the barycentric subdivision of $X$ and used to deduce connectivity properties of subsets of $X$.

For any cell $\tau \in X$ we let ${\text{dim}}(\tau)$ be its dimension. There is also a number defined in \cite{B-G-W} as ${\text{dp}}(\tau)$ which refers to the ``depth" of a cell. We refer the reader to Section 8 of \cite{B-G-W} for the definition of the depth of a cell. 

We let  $\mathring{X}$ be the barycentric subdivision of the Euclidean building $X$.  For any cell $\tau \se X$, we let $\mathring{\tau}$ be its barycenter. Bux-K\"{o}hl-Witzel assigned to $\mathring{\tau}$ the triple of real numbers

$$f_{BKW}(\mathring{\tau})=\big(\max_{x \in \tau}(h(x)), {\text{dp}} (\tau), {\text{dim}}(\tau)\big)$$ The function $f_{BKW}$ is a combinatorial Morse function when triples of real numbers are ordered lexicographically.

For any triple of real numbers $s$ that is greater than or equal to the triple $s_0=(1,0,0)$, we let $\mathring{X}(s) $ be the subcomplex of $ \mathring{X}$ spanned by the $\mathring{\tau}$  for which $f_{BKW}(\mathring{\tau}) \leq s$. Since $f_{BKW}$ is ${\bf G}(\mathcal{O}_S)$-invariant, so to is $\mathring{X}(s)$. Since $\mathring{X}(s)$ is a closed subset of $\mathring{X}$ whose height is bounded, it is cocompact modulo ${\bf G}(\mathcal{O}_S)$. The values of $f_{BKW}$ are finite below any given bound, and we let $s+1$ denote the least value of $f_{BKW}$ that is greater than $s$.

 We let ${\rm Lk}(\mathring{\tau})$ be the link of $\mathring{\tau}$ in $\mathring{X}$, and we define the \emph{Morse descending link} of $\mathring{\tau}$ with respect to the Morse function $f_{BKW}$  to be the complex of simplices $\sigma \se{\text{Lk}}(\mathring{\tau})$ such that $f_{BKW}(v)<f_{BKW}(\mathring{\tau})$ for every vertex $v \in \sigma$. To obtain $\mathring{X}(s+1)$ we attach to $\mathring{X}(s)$ the descending links of cells $\mathring{\tau} \se \mathring{X}$ with $f_{BKW}(\mathring{\tau} )=s+1$. The work of Bux-K\"{o}hl-Witzel is to have defined $f_{BKW}$ in such a way as to utilize the work of Schulz \cite{Sc} in showing that the Morse descending links of vertices in $\mathring{X}$  are either contractible or spherical of dimension $(k-1)$. Thus, up to homotopy equivalence, $\mathring{X}(s+1)$ is obtained by attaching $k$-cells to $\mathring{X}(s)$. This process induces an isomorphism of homotopy groups $\pi_i(\mathring{X}(s))\cong \pi_i (\mathring{X}(s+1))$ for $i \leq k-2$.
 Since ${X}$ is contractible and the union of the $\mathring{X}(s)$, we have that  $\mathring{X}(s)$ is $(k-2)$-connected for any $s \geq s_0$. It is the existence of a ${\bf G}(\mathcal{O}_S)$-cocompact $(k-2)$-connected space that can be viewed as the main result of \cite{B-G-W} as it immediately implies  that ${\bf G}(\mathcal{O}_S)$ is of type $F_{k-1}$.

  In what remains, we will let $X_{k-2}=X(s_0)$.  In particular, $X_{k-2}$ is a $(k-2)$-connected subcomplex of $X$ that is invariant and cocompact under the action of  ${\bf G}(\mathcal{O}_S)$. 
 We will also pass to a subsequence of the $y_n$ to assume that $S_n\cap X_{k-2} = \emptyset$ for all $n$.

The following lemma demonstrates the  compatibility of $\beta_{P}$ and $f_{BKW}$ on $S_n$.

\begin{lemma}\label{l:same} The Morse descending link of $y_n$ with respect to $f_{BKW}$ equals $BS^\downarrow _n$.
\end{lemma}

\begin{proof}
As in Section 6 of \cite{B-G-W}, the height function $h$ forces a decomposition of the link of $y_n \in X$ into a join of a ``horizontal link" of $y_n$ and a ``vertical link"  of $y_n$ where the horizontal link of $y_n$ is the join of all factors of the link of $y_n$ whose points are evaluated by $h$ as  $h(y_n)$.

By Lemma~\ref{l:yn}, the restriction of $\beta_{P}$ to the horizontal link of $y_n$ is constant. But $y_n$ is a special vertex, so Proposition~\ref{p:factores} implies that the horizontal link of $y_n$ is trivial, and therefore, that the vertical link of $y_n$ equals the link of $y_n$.

Now by Proposition 9.6 of \cite{B-G-W}, the Morse descending link of $y_n$ is the subcomplex of the link of $y_n$ in $X$ that is spanned by all vertices $v$ in the link of $y_n$ such that $h(v)<h(y_n)$. (Keep in mind that any vertex of $X$ is ``significant".) Again, by Lemma~\ref{l:yn}, this complex is equal to $BS^\downarrow _n$.
\end{proof}

\subsection{Extending local disks near $y_n$}
In addition to the existence of $X_{k-2}$, we shall utilize the results of \cite{B-G-W} to extend ``local" disks near $y_n$ to ``global" disks in $X$. More precisely, we have

\begin{lemma}\label{l:extendstar} Let $\sigma : S^{k-1} \rightarrow X$ be a continuous map of a $(k-1)$-sphere into $X$.  Suppose there is some triple $s > s_0$ such that $\sigma(S^{k-1}) \se \mathring{X}(s)$. Then there is a homotopy $F:S^{k-1} \times [0,1] \rightarrow X$ such that for all $x \in S^{k-1}$ we have $F(x,t)\in \mathring{X}(s)$, $F(x,0)=\sigma(x)$, and $F(x,1)\in  \mathring{X}(s_0)=X_{k-2}$.
\end{lemma}

\begin{proof}
Let $c_1^0, \ldots,c_m^0 \se X$ be the image under $\sigma$ of the $0$-cells of $S^{k-1}$. Let $c^0_{i,F} \se \mathring{X}(s)$ be paths from $c_i^0$ to $X_{k-2}$. The boundary of each $c^0_{i,F}$ is $c_i^0$ and $b_i^0$ for some $b_i^0 \in X_{k-2}$.

If $k=1$, then $m= 2$, and $c^0_{1,F} \cup c^0_{2,F}$ is the image of the homotopy $F$.

If $k \geq 2$, then let $c_i^1 \se \sigma( S^{k-1})$ be the image of the 1-cell with boundary $c_\ell^0$ and $c_j^0$. Since $\mathring{X}(s)$ is obtained from $X_{k-2}$ by attaching $k$-cells, there is a homotopy relative $b_\ell^0$ and $b_j^0$ between $c_i^1 \cup c^0_{\ell,F}\cup c^0_{j,F}$ and a 1-cell $b_i^1 \se X_{k-2}$. We name the image of this homotopy $c_{i,F}^1$.

If $k=2$, then the union of the $c^1_{i,F}$ defines the homotopy $F$.

If $k \geq 3$, then we proceed as above by induction on the skeleta of $S^{k-1}$.
\end{proof}

We let $I_n= S_n - \partial S_n$ be the interior of $S_n$. As a consequence of the above lemma, we have 

\begin{corollary}\label{c:extend} 
For  $n \gg 0$, there is a $k$-disk $D^k_n \se S^\downarrow _n \cup (X-{\bf G}(\mathcal{O}_S) I_n)$ with $\partial D_n^k \se X_{k-2}$ and such that $D^k_n \cap S^\downarrow _n$ is a $k$-disk that represents a noncontractible $k$-sphere in the quotient space $S^\downarrow _n /BS^\downarrow _n$.
\end{corollary}

\begin{proof}
Let $s_n$ be the triple such that $f_{BKW}(y_n)=s_n$. By Lemma~\ref{l:yn}, and the definition of the Morse function $f_{BKW}$, we have for any cell $\tau \se S_n$ that is not contained in $\partial S_n$ that $f_{BKW}({\bf G}(\mathcal{O}_S) \tau)=f_{BKW}(\tau)\geq s_n$ since $y_n \in \tau$. That is, ${\bf G}(\mathcal{O}_S) I_n \cap  \mathring{X}(s_n-1)=\emptyset$.

By Lemmas~\ref{l:Schulz} and~\ref{l:same}, there is a noncontractible $(k-1)$-sphere $\sigma^{k-1}_n \se 
B S_n^\downarrow$. We let $d^k_n \se S_n^\downarrow$ be the cone at $y_n \in S_n^\downarrow$ on 
 $$\sigma^{k-1}_n \se B S_n^\downarrow \se \mathring{X}(s_n-1)$$

By Lemma~\ref{l:extendstar}, there is a homotopy $F$ between $\partial d^k_n$ and a $(k-1)$-sphere in $X_{k-2}$ whose image is contained in $\mathring{X}(s_n-1)$. We let $D^k_n$ be the union of $d^k_n$ and $F$. Then $$D^k_n \se S^\downarrow _n \cup \mathring{X}(s_n-1) \se S^\downarrow _n \cup (X-{\bf G}(\mathcal{O}_S) I_n)$$

That $D^k_n \cap S^\downarrow _n=d^k_n$ represents a noncontractible $k$-sphere in $S^\downarrow _n /BS^\downarrow _n$ follows from the natural identification of $d^k_n / \partial d^k_n$ and  $S^\downarrow _n /BS^\downarrow _n$ with the suspensions of $\sigma^{k-1}_n$ and 
$BS^\downarrow _n$ respectively.
\end{proof}

\begin{lemma} \label{l:pixar}
Suppose that $\mathfrak{C}_a, \mathfrak{C}_b \se S^\downarrow _n$ are chambers in $X$, and that there is some  $\gamma  \in {\bf G}(\mathcal{O}_S) $ such that $\gamma \mathfrak{C}_a= \mathfrak{C}_b$. Then $\gamma y_n=y_n$.
 \end{lemma}

\begin{proof}
The vertex $y_n$ is the only vertex of any chamber in $S^\downarrow _n$ with $f_{BKW}(v)=f_{BKW}(y_n)$. Since $f_{BKW}$ is ${\bf G}(\mathcal{O}_S)$ invariant, we have for $\gamma y_n \in \mathfrak{C}_b$ that $f_{BKW}(\gamma y_n)=f_{BKW}(y_n)$ so that $\gamma y_n=y_n$.
\end{proof}

 \section{Construction of a $k$-connected ${\bf G}(\mathcal{O}_S)$-complex}\label{s:space}
Bux-K\"{o}hl-Witzel gives us a $(k-2)$-connected complex that ${\bf G}(\mathcal{O}_S)$ acts on properly and cocompactly, namely $X_{k-2}$. In order to determine the cohomology of finite-index subgroups of ${\bf G}(\mathcal{O}_S)$ in dimension $k$, we will create a $k$-connected space that ${\bf G}(\mathcal{O}_S)$ acts on. In this section we will construct such a space by attaching $k$-cells to $X_{k-2}$ and then attaching $(k+1)$-cells after that.

\subsection{Construction of $X_{k}$.}

We let $\psi:X_{k-2} \rightarrow X$ be the inclusion. In the process of our construction of a $k$-connected space that contains $X_{k-2}$, we will be extending $\psi$ to a map from that $k$-connected space into $X$.

Let $\sigma : S^{k-1}  \rightarrow X_{k-2}$ be a continuous map of a $(k-1)$-sphere into the $(k-1)$-skeleton of $X_{k-2}$. We regard $\sigma$ as an attaching map for a $k$-cell that we name $D^k_{1,\sigma}$.

For each nontrivial $\gamma \in {\bf G}(\mathcal{O}_S)$, we attach another $k$-cell $D^k_{\gamma,\sigma}$ to $X_{k-2}$ using the attaching map $\gamma \circ \sigma$. We assign a homeomorphism $\gamma : D^k_{1,\sigma} \rightarrow D^k_{\gamma,\sigma}$ that restricts to the $\gamma$-action on $\partial D^k_{1,\sigma}, \partial D^k_{\gamma,\sigma} \se X_{k-2}$. Then for any  $\lambda \in {\bf G}(\mathcal{O}_S)$, we let $$\lambda : D^k_{\gamma,\sigma} \rightarrow  D^k_{\lambda \gamma,\sigma}$$ be the homeomorphism defined by $\lambda = (\lambda \gamma)\gamma ^{-1}$. In this way, we have defined a ${\bf G}(\mathcal{O}_S)$-action on the complex 
 $$ X_{k-2} \cup \bigcup_{\gamma \in {\bf G}(\mathcal{O}_S) }D^k_{\gamma,\sigma} $$

We repeat the process above for every continuous $\sigma : S^{k-1}  \rightarrow X_{k-2}$ with image in the $(k-1)$-skeleton of $X_{k-2}$. The resulting union of $X_{k-2}$ with the union of every $D^k_{\gamma,\sigma}$ for every pair of $\gamma$ and $\sigma$ is a $k$-complex that we will denote by $X_{k-1}$. Notice that $X_{k-1}$ is a $(k-1)$-connected,  ${\bf G}(\mathcal{O}_S)$-complex. The group  ${\bf G}(\mathcal{O}_S)$ will not in general act freely on $X_{k-1}$, but any nontrivial point stabilizers correspond to points in $X_{k-2}$ since the interiors of each of the $D^k_{\gamma,\sigma}$ are disjoint.

We extend $\psi$ to each $D^k_{\gamma,\sigma}$ --- and thus to all of $X_{k-1}$ --- by assigning arbitrary continuous maps $\psi : D^k_{1,\sigma} \rightarrow X$  that agree with $\psi$ on $\partial D^k_{1,\sigma} \se X_{k-2}$ and then by defining $\psi : D^k_{\gamma,\sigma} \rightarrow X$ as $\gamma \circ \psi \circ \gamma ^{-1}$. Notice that $\gamma \circ \psi = \psi \circ \gamma$ so that $\psi$ is ${\bf G}(\mathcal{O}_S)$-equivariant.

Now repeat the above process, this time attaching $(k+1)$-cells $D^{k+1}_{\gamma,\sigma}$ to $X_{k-1}$ with attaching maps $\sigma : S^{k}  \rightarrow X_{k-1}$
 to obtain a $k$-connected complex $X_k$ that ${\bf G}(\mathcal{O}_S)$ acts on with a ${\bf G}(\mathcal{O}_S)$-equivariant map $\psi :X_k \rightarrow X$ that restricts to $X_{k-2}\se X$ as the inclusion map. The action of ${\bf G}(\mathcal{O}_S)$ on $X_{k} - X_{k-2}$ is free.

\section{Assigning attaching disks to cycles in a finite complex}
In this section we will begin to focus some attention on a given finite-index subgroup $\Gamma$ of ${\bf G}(\mathcal{O}_S)$ from the statement of our main result, Theorem~\ref{t:mt}. That is, we let $\Gamma$ be any finite-index subgroup of ${\bf G}(\mathcal{O}_S)$ that is residually $p$-finite.

Our goal in proving our main result is to show that  $\emph{H}^k(\Gamma \backslash X_k;\mathbb{F}_p)$ is infinite. In the penultimate section of this paper we explain why this implies that $ \emph{H}^k(\Gamma ;\mathbb{F}_p)$ is infinite.

\subsection{Definition of $\Gamma _n$} Our proof of our main result relies on forming a sequence of finite quotients of the group $\Gamma$. These quotients are described in the following

\begin{lemma}\label{l:ru4}
For any $n \geq 0$, there is a normal subgroup $\Gamma _n  \unlhd \Gamma $ such that  $\Gamma / \Gamma _n$ is a finite $p$-group and $\Gamma _n $ acts cocompactly and freely on $\Gamma S_{n}$.
\end{lemma}

\begin{proof}
The group $\Gamma$ acts cocompactly on $\Gamma S_n$.

For any cell $\tau \se S_n$, let $\Gamma_\tau$ be the finite stabilizer of $\tau$ in $\Gamma$, and let $Z_n \se \Gamma$ be the finite set of the union of $\Gamma _\tau$ over the finite set of cells $\tau \se S_n$. 

Since $\Gamma$ is residually $p$-finite, there is for each nontrivial $\gamma \in Z_n$ a finite $p$-group, $G_\gamma$, and a homomorphism $\phi _\gamma : \Gamma \rightarrow G_\gamma$ such that $\phi_\gamma (\gamma)\neq 1$. Now let $\phi : \Gamma \rightarrow \prod_\gamma G_\gamma$ be the product of the $\phi_\gamma$, and let $\Gamma _n$ be the kernel of $\phi$. Then $\Gamma _n \unlhd \Gamma$,  $\Gamma / \Gamma _n$ is a finite $p$-group, and $Z_n \cap \Gamma _n =\{1\}$.

Since $\Gamma_n$ is finite-index in $\Gamma$, it acts cocompactly on $\Gamma S_n$. Furthermore, if $\gamma \in \Gamma _n$ and $\gamma g \tau =g \tau$ for some $g \in \Gamma$ and some cell $\tau \se S_n$, then $g^{-1} \gamma g \in \Gamma _n$ is contained in $\Gamma _\tau \se Z_n$, and thus  $g^{-1} \gamma g$, and hence $\gamma$, is trivial. \end{proof}

\subsection{Definition of $\theta_n$} We define $$\theta _n : X \rightarrow \Gamma_n \backslash X$$
to be the quotient map.  Notice that $\Gamma $ acts on $\Gamma_n \backslash X$ since $\Gamma_n$ is normal in $\Gamma $. Furthermore,  $\theta _n$ is   $\Gamma$-equivariant.

Also note that $\Gamma$ acts on the pair  $( X, X-\Gamma I_n)$ and thus on 
the pair $( \theta_n(X), \theta_n(X-\Gamma I_n))$, and therefore on the homologies of these pairs. (All homologies of complexes in this paper are cellular.)

\subsection{Definition of  $\Theta _n ( D_{\gamma,\sigma}^k)$}
Given a  $k$-cell $D_{\gamma,\sigma}^k$ attached to $X_{k-2}$ in the construction of $X_k$, 
we have that $\psi(\partial D_{\gamma,\sigma}^k) \se X_{k-2}$.

By Lemma~\ref{l:yn}, the sequence of $h(y_n)$, and hence of $f_{BKW}(\Gamma y_n)$ is unbounded. Thus we may assume that  $X_{k-2}$ intersects each $\Gamma S_n$ trivially, which implies  $\partial \psi( D_{\gamma,\sigma}^k) \se X-\Gamma I_n$ and thus that $\psi( D_{\gamma,\sigma}^k) $ represents a class in the homology group $H_k( X, X-\Gamma I_n ; \mathbb{F}_p)$, and further, that $\theta_n \circ \psi ( D_{\gamma,\sigma}^k) $ represents a class in the homology group $H_k( \theta_n(X),\theta_n( X-\Gamma I_n) ; \mathbb{F}_p)$. In the remainder we shall let 
 $$\Theta _n ( D_{\gamma,\sigma}^k)= [\theta _n \circ \psi (D_{\gamma,\sigma}^k)] \in H_k( \theta_n(X),\theta_n( X-\Gamma I_n) ; \mathbb{F}_p)$$
 
Recall that $\psi$ is $\Gamma$-equivariant, and that $\theta_n$ is $\Gamma$-equivariant. Therefore, the group $\Gamma $ acts on the set of all $\Theta _n ( D_{\gamma,\sigma}^k)$ by the rule that 
 if $g \in \Gamma $, then 
\begin{align*}g \Theta_n ( D_{\gamma,\sigma}^k)&= g [\theta _n \circ \psi (D_{\gamma,\sigma}^k)] \\ & =[\theta _n \circ \psi (g D_{\gamma,\sigma}^k)] \\ &=[\theta _n \circ \psi ( D_{g\gamma,\sigma}^k) ] \\ &=\Theta _n ( D_{g\gamma,\sigma}^k)  \end{align*}

\subsection{Definition of  $W_n$} We let $W_n$ be the vector subspace of \newline $H_k( \theta_n(X),\theta_n( X-\Gamma I_n) ; \mathbb{F}_p)$ generated by the classes $\Theta _n ( D_{\gamma,\sigma}^k)$ for every pair $\gamma$ and $\sigma$.

By the above, the $\Gamma$-action on $H_k( \theta_n(X),\theta_n( X-\Gamma I_n) ; \mathbb{F}_p)$ restricts to a $\Gamma$-action on $W_n$. Since $\Gamma_n$ acts trivially on $\theta_n(X)$, the action of $\Gamma$ on $W_n$ factors through the finite $p$-group $\Gamma / \Gamma_n$.

\begin{lemma}\label{l:ducks}
The vector space $W_n$ is finite-dimensional and nonzero.
\end{lemma}

\begin{proof}
The space $X$ is the union of $\Gamma S_n$ and $X-\Gamma I_n$, so $\Gamma S_n$ surjects via $\theta_n$ onto the quotient $\theta_n(X)/\theta_n( X-\Gamma I_n)$. Lemma~\ref{l:ru4} gives us that $\theta_n(\Gamma S_n)$ is a finite complex, and thus,  $\theta_n(X)/\theta_n( X-\Gamma I_n)$ is finite.
The
finite dimensionality of $W_n$ now follows from the finite dimensionality of $H_k( \theta_n(X),\theta_n( X-\Gamma I_n) ; \mathbb{F}_p)$.

Let $D_n^k \se X$ be as in Corollary~\ref{c:extend}. We claim that $\theta_n(D_n^k)$ represents a nonzero class in $H_k( \theta_n(X),\theta_n( X-\Gamma I_n) ; \mathbb{F}_p)$. Indeed, $BS^\downarrow _n 
\se X-\Gamma I_n$ and
it suffices to prove that $$(\theta_n)_* : H_k( S^\downarrow _n , BS^\downarrow _n ; \mathbb{F}_p) \longrightarrow H_k( \theta_n(X),\theta_n( X-\Gamma I_n) ; \mathbb{F}_p)$$ is injective. As  $\theta_n(X)$ is a $k$-dimensional complex, this reduces to showing that $\theta_n(\mathfrak{C}_a) \neq \theta_n(\mathfrak{C}_b)$ for distinct chambers $\mathfrak{C}_a, \mathfrak{C}_b \se S^\downarrow _n$. In other words, we want to show that  $\gamma \mathfrak{C}_a= \mathfrak{C}_b$ for any $\gamma \in \Gamma _n$ and any pair of chambers $\mathfrak{C}_a, \mathfrak{C}_b \se S^\downarrow _n$ implies that $\mathfrak{C}_a= \mathfrak{C}_b$. By Lemma~\ref{l:pixar}, any such $\gamma \in \Gamma _n$ fixes $y_n\in \Gamma S_n$, and by Lemma~\ref{l:ru4}, $\gamma$ is trivial so that $\mathfrak{C}_a= \mathfrak{C}_b$.

Now let $\sigma_n :S^{k-1} \rightarrow X_{k-2}$ represent $\partial D_n^k$, and let $D^k_{1,\sigma_n}$ be the $k$-disk attached to $X_{k-2}$ by $\sigma_n$ in the construction of $X_k$. Since $X$ is contractible and $k$-dimensional, and since $D_n^k$ and $\psi(D_{1,\sigma _n}^k)$
share a common boundary, they represent the same $k$-chain in the homology of $X$. Therefore, by the above paragraph, $$\Theta _n ( D_{1,\sigma_n}^k)  =[\theta _n \circ \psi(D_{1,\sigma _n}^k)] =[ \theta_n(D^k_n)] $$ is a nonzero class in $W_n \leq H_k( \theta_n(X),\theta_n( X-\Gamma I_n) ; \mathbb{F}_p)$.  
\end{proof}

\section{A sequence of cycles and cocycles for $\Gamma \backslash X_k$}
The action of  $\Gamma $ on  $W_n$ induces an action of $\Gamma $ on  the dual vector space $W_n^*$
 by $\gamma \phi (x)=\phi (\gamma^{-1} x) $ for $\gamma \in \Gamma $, $\phi \in W_n^*$, and 
 $x \in W_n$.

\begin{lemma}\label{l:asin}
For each $n$, there is  a $\Gamma $-invariant $\varphi_n \in W_n^*$
and some $\lambda _n \in {\bf G}(\mathcal{O}_S)$ and $\tau_n : S^{k-1} \rightarrow X_{k-2}$ such that $\varphi_n(\Theta_n(D_{\lambda_n,\tau_n}^k))\neq 0$.  Furthermore, after passing to a subsequence, if $m >n$ then $\varphi_m ( \Theta_m(D_{\lambda _n ,\tau_n}^k))= 0$.
\end{lemma}

\begin{proof}
A linear transformation of a finite-dimensional nonzero vector space of characteristic $p$ is unipotent if and only if it has order $p^k$ for some $k$. Since the action of $\Gamma$ on $W^*_n$ factors through the $p$-group $\Gamma / \Gamma_n$, the elements of $\Gamma$ 
 act  on $W_n^*$ as unipotent transformations. By Kolchin's Theorem, any group of unipotent transformations on a finite-dimensional nonzero vector space fixes  a nonzero vector. That is, there is some $\Gamma$-invariant ${\varphi}_n \in W_n^*$ and some $k$-disk $D^k_{\lambda _n ,\tau _n}$ from the construction of $X_k$ such that ${\varphi}_n(\Theta_n(D^k_{\lambda _n ,\tau _n}  )) \neq 0$.
 
 Given the disk $D^k_{\lambda _n ,\tau _n}$ above, we may assume that the $f_{BKW}$-values of the cells in $S_{n+1}$, and hence of those in $\Gamma S_{n+1}$ exceed the $f_{BKW}$-values of the finitely many cells in $\psi( D^k_{\lambda _n ,\tau _n})$. Thus, if $m>n$ we have that $\psi( D^k_{\lambda _n ,\tau _n})\se X-\Gamma I_m$ and thus $\Theta_m( D_{\lambda _n,\tau_n}^k)=0$ in $W_m$.
    \end{proof}

\subsection{Cocycles}

Let $D^{k}_{\gamma,\sigma}$ be a $k$-cell that was attached to $X_{k-2}$ in the construction of $X_k$. Recall that $\Theta_n ( D_{\gamma,\sigma}^k)$ represents a class in $W_n$ and that $\varphi _n$ is a $\Gamma$-invariant functional on $W_n$.

\begin{lemma}\label{l:alphau7}
For any $n \geq 0$, $\gamma \in {\bf G}(\mathcal{O}_S)$, $g \in \Gamma$, and $D^{k}_{\gamma,\sigma}$, we have $\varphi_n ( \Theta _n (D^{k}_{\gamma,\sigma}))=\varphi_n ( \Theta _n (gD^{k}_{\gamma,\sigma}))$.
\end{lemma} 

\begin{proof}
This is immediate since  $\psi$ is $\Gamma$-equivariant, $\theta_n$ is $\Gamma$-equivariant, and $\varphi_n$ is $\Gamma$-invariant.
\end{proof}

Let $q : X_k \rightarrow \Gamma \backslash X_k$ be the quotient map. 
Note that any $k$-cell in $\Gamma \backslash X_k$ is contained in $\Gamma \backslash X_{k-2}$ or else is of the form $q(D^{k}_{\gamma ,\sigma})$ for some $D^{k}_{\gamma ,\sigma} \se X_k$.
We define the $k$-cochain  ${\Phi}_n$ on $k$-chains in $ \Gamma \backslash X_k$ with values in  $\mathbb{F}_p $ as $0$ on $\Gamma \backslash X_{k-2}$ and $$\Phi_n(q(D^{k}_{\gamma,\sigma}))= \varphi_n ( \Theta _n (D^{k}_{\gamma,\sigma}))$$ for any $q(D^{k}_{\gamma ,\sigma})$, and then we extend linearly. The previous lemma tells us that ${\Phi}_n$ is well-defined.

\begin{lemma}
${\Phi}_n$ is a cocycle.
\end{lemma}

\begin{proof}

The $(k+1)$ cells of $ \Gamma \backslash X_k$ are of the form $q( D^{k+1}_{\gamma,\sigma})$, so we must check that ${\Phi}_n$ evaluates the boundary of any
$q( D^{k+1}_{\gamma,\sigma})$ trivially.

Let $\mathfrak{C}_1,\ldots, \mathfrak{C}_m$ be a collection of $k$-cells in $X_{k-2}$ such that the chain $ \partial D^{k+1}_{\gamma,\sigma}$ equals $\sum_j \mathfrak{C}_j + \sum_i D^k_{\gamma _i , \sigma_i}$ for some $D^k_{\gamma_i, \sigma_i}$ where we suppress in this notation the orientation of $k$-cells. Then $\partial q(D^{k+1}_{\gamma,\sigma}) = \sum_j q( \mathfrak{C}_j) + \sum_i q(  D^k_{\gamma _i , \sigma_i})$.

Note that $\psi ( \partial D^{k+1}_{\gamma,\sigma})$ is a $k$-sphere in the $k$-dimensional and contractible $X$, and hence it represents the 0-chain. That is, the chain $\psi ( \sum_j \mathfrak{C}_j + \sum_i D^k_{\gamma _i , \sigma_i})\cap \Gamma  S_n$, and hence $\psi( \sum_i D^k_{\gamma _i , \sigma_i}) \cap \Gamma  S_n$, is the 0-chain. Therefore,  $ \Theta _n ( \sum_i D^k_{\gamma _i , \sigma_i})$ is the 0-chain, which implies
\begin{align*} {\Phi}_n\big(\partial q(D^{k+1}_{\gamma,\sigma})\big)&={\Phi}_n\Big(\sum_j q( \mathfrak{C}_j) + \sum_i q(D^k_{\gamma _i , \sigma_i})\Big) \\
& =\Phi_n\Big( \sum_i q(D^k_{\gamma _i , \sigma_i})\Big) \\
& = \varphi_n \circ \Theta _n \Big( \sum_i D^{k}_{\gamma_i,f_i}\Big) \\
& =\varphi_n ( 0 )  \\
&=0
\end{align*}
\end{proof}

\subsection{Cycles}
Given $ D_{\lambda _n,\tau_n}^k$ as in Lemma~\ref{l:asin}, the $k$-chain $ D_{\lambda _n,\tau_n}^k -    D_{\lambda _0,\tau_0}^k $ is the difference of two $k$-disks in $X_k$. We let $$C_n = q \big(  D_{\lambda _n,\tau_n}^k )- q(   D_{\lambda _0,\tau_0}^k  \big)$$ which is a $k$-chain in $\Gamma \backslash X_k$.

\begin{lemma}\label{l:qcycles}
After passing to a subsequence in $n$, each $C_n$ is a $k$-cycle over $\mathbb{F}_p$ in $\Gamma \backslash X_k$.
\end{lemma}

\begin{proof}
Notice that $q(\partial D_{\gamma _n,\sigma_n}^k)$ is a $(k-1)$-cycle in $\Gamma \backslash X_{k-2}$. Since $\Gamma \backslash X_{k-2}$ is compact, there are only finitely many cellular $(k-1)$-chains in $\Gamma \backslash X_{k-2}$ with coefficients in $\mathbb{F}_p$. Therefore, we may pass to a subsequence and assume that $q(\partial  D_{\lambda _n,\tau_n}^k)$ is a constant $\mathbb{F}_p$-cycle for $n \geq 0$.
\end{proof}

We can now prove

\begin{proposition}\label{p:mp}
$H^k(\Gamma \backslash X_k ; \mathbb{F}_p)$ and $H_k(\Gamma \backslash X_k ; \mathbb{F}_p)$ are infinite.
\end{proposition}

\begin{proof}
Let $m \geq n > 0$. By the definitions of $\Phi_n$ and $C_n$, and by Lemma~\ref{l:asin}, 
\begin{align*}\Phi_m(C_n)
&=\Phi_m \big(q(  D_{\lambda _n,\tau_n}^k ) \big)- \Phi_m \big( q(   D_{\lambda _0,\tau_0}^k  ) \big) \\ 
&=\varphi_m \big(\Theta_m (  D_{\lambda _n,\tau_n}^k ) \big)- \varphi_m  \big(\Theta_m(   D_{\lambda _0,\tau_0}^k  ) \big) \\  
&=\varphi_m(\Theta_m \big(  D_{\lambda _n,\tau_n}^k ) \big)
\end{align*} 
does not equal 0 if $m=n$, but does equal 0 if  $m>n$. Thus, each of the terms in the sequences $[\Phi_n]\in H^k(\Gamma \backslash X_k ; \mathbb{F}_p)$ and $[C_n]\in H_k(\Gamma \backslash X_k ; \mathbb{F}_p)$ are distinct.
\end{proof}

\section{Proof of Theorem~\ref{t:mt}}\label{s:end}

If $\Gamma$ acts freely on $X_k$, then Theorem~\ref{t:mt} is immediate from Proposition~\ref{p:mp}. And one can always choose a finite-index, residually $p$-finite subgroup of ${\bf G}(\mathcal{O}_S)$ that acts freely on $X_k$ (see the following section). However, to show Theorem~\ref{t:mt} holds for any, and not just some, finite-index, residually $p$-finite subgroup of ${\bf G}(\mathcal{O}_S)$, we need to apply one more technique. That is the goal of this section.

By our construction of $X_k$, the group $\Gamma$ acts freely on $X_k -X_{k-2}$, and while it may not be true that $\Gamma$ acts freely on $X_{k-2}$, it does act cocompactly on  $X_{k-2}$. That is, there are only finitely many $k$-cells in the quotient  $\Gamma \backslash  X_{k-2}$.  This will imply Theorem~\ref{t:mt} after the application of a spectral sequence.

The material from this section is taken from Chapter VII of Brown's text on Cohomology of Groups \cite{Brown}.

We begin by subdividing $X_k$ such that individual cells in $X_k$ inject into $\Gamma \backslash X_k$.

We let $H^\Gamma _k (X_k ; \mathbb{F}_p)$ be the $k$-th equivariant homology group of $\Gamma$ and $X_k$ with coefficients in $\mathbb{F}_p$. That is, if $C_*(X_k;\mathbb{F}_p)$ is the chain complex for the homology of $X_k$ with coefficients in $\mathbb{F}_p$, and if $F_*$ is a projective resolution of $\mathbb{Z}$ over $\mathbb{Z}\Gamma$, then 
$$H^\Gamma _k (X_k ; \mathbb{F}_p)=H_k(F_* \otimes _\Gamma C_*(X_k;\mathbb{F}_p) ) $$

\begin{lemma}
$H^\Gamma _k (X_k ; \mathbb{F}_p)=H _k (\Gamma ; \mathbb{F}_p)$
\end{lemma}

\begin{proof}
The complex $F_* \otimes _\Gamma C_*(X_k;\mathbb{F}_p)$ is a double complex with  an associated spectral sequence  $$E^1_{\ell,q}=H_q(F_\ell \otimes _\Gamma C_*(X_k;\mathbb{F}_p))=F_\ell \otimes _\Gamma H_q(X_k;\mathbb{F}_p)$$
and  
$$E^2_{\ell,q}=H_\ell(\Gamma;H_q(X_k;\mathbb{F}_p))$$

Notice that if $0<q\leq k$ then $E^2_{\ell,q}=H_\ell(\Gamma ; 0)=0$ since $X_k$ is $k$-connected. It follows that $E^r_{\ell,q}=0$ when $r \geq 2$ and $0<q\leq k$. Hence, $$H_k(\Gamma;\mathbb{F}_p)=E^2_{k,0}=E^\infty_{k,0} = \bigoplus _{\ell+q=k}E^\infty_{\ell,q} $$ The lemma follows since the spectral sequence converges to $H^\Gamma _* (X_k ; \mathbb{F}_p)$.
\end{proof}

The complex $F_* \otimes _\Gamma C_*(X_k;\mathbb{F}_p)$ is also a double complex with an associated spectral sequence where $E^1_{\ell,q}=H_q(F_* \otimes _\Gamma C_\ell(X_k;\mathbb{F}_p))$.  The spectral sequence converges to $H^\Gamma _* (X_k ; \mathbb{F}_p)$, and in particular, $$H_k(\Gamma  ; \mathbb{F}_p)=H^\Gamma _k (X_k ; \mathbb{F}_p)= \bigoplus _{\ell+q=k}E^\infty_{\ell,q}$$

As in VII.7.7 of \cite{Brown}, $$E^1_{\ell,q} = \bigoplus _{c \in Y_\ell}H_q(\Gamma _c ; \mathbb{F}_p)$$ where $Y_\ell$ is a set of representatives of $\ell$-cells in $X_k$ modulo $\Gamma$, and $\Gamma _c$ is the stabilizer in $\Gamma$ of $c$.

\begin{lemma}
If $r,q \geq 1$, then $E^r_{\ell,q}$ is finite.
\end{lemma}

\begin{proof}
Since $\Gamma$ acts cocompactly on $X_{k-2}$ and freely on $X_k - X_{k-2}$, there are only finitely many $c \in Y_\ell$ such that $\Gamma _c \neq 1$. Thus, 
$E^1_{\ell,q}$ is finite as it is a finite sum of homology groups of finite groups with coefficients in a finite field. The lemma follows since the dimension of
$E^r_{\ell,q}$ is bounded by that of 
$E^1_{\ell,q}$.
\end{proof}

\begin{lemma}
$E^2_{\ell,0}=H_\ell(\Gamma \backslash X_k ; \mathbb{F}_p)$. In particular, by Proposition~\ref{p:mp}, $E^2_{k,0}$ is infinite.
\end{lemma}

\begin{proof}
Let $\partial '$ be the boundary operator for $C_*(X_k;\mathbb{F}_p)$, and for any $(\ell-1)$-cell $d \se X_k$, let $\pi _d$ be the projection of $C_{\ell -1}(X_k;\mathbb{F}_p)$ onto the coordinate represented by $d$.

We let $\partial$ be the boundary operator for the chain complex of $\Gamma \backslash X_{k}$, denoted as $C_*(\Gamma \backslash X_k;\mathbb{F}_p)$.

Notice that $E^2_{*,0}$ is the homology of the complex $(E^1_{k,0}, d^1)$ where $d^1: E^1_{\ell,0} \rightarrow E^1_{\ell-1,0}$. There is a natural identification of $$E^1_{\ell,0}= \bigoplus _{c \in Y_\ell}H_0(\Gamma _c ; \mathbb{F}_p)=\bigoplus _{c \in Y_\ell}\mathbb{F}_p$$ with $$C_\ell (\Gamma \backslash X_k;\mathbb{F}_p)$$ given by $$(a_c)_{c\in Y_\ell} \mapsto \sum_{\Gamma c \se \Gamma \backslash X_k} a_c (\Gamma c)$$ where $a_c \in \mathbb{F}_p$. Below we apply this identification liberally.

Our goal is to show that $d^1$ can be identified with $\partial$. For this, if $c \in Y_\ell$ then we let $\mathcal{D}_c$ be the set of $(\ell -1)$-cells in $X_k$ contained in $c$. Then VII.8.1 of \cite{Brown} tells us that if $a_c \in \mathbb{F}_p = H_0 (\Gamma _c ; \mathbb{F}_p)$ then, up to sign,
$$d^1(a_c)=\sum _{d \in \mathcal{D}_c}v_d \circ u_{cd} \circ t_c (a_c) $$
where $t_c :H_0 (\Gamma _c ; \mathbb{F}_p) \rightarrow H_0(\Gamma _c ; \mathbb{F}_p)$ is transfer --- and thus is the identity --- and where $v_d : H_0 (\Gamma _d ; \mathbb{F}_p) \rightarrow  H_0 (\Gamma _{d_0} ; \mathbb{F}_p)$ for $d_0 \in Y_{\ell -1}$ is such that $\Gamma d = \Gamma d_0$ and $v_d$ is induced by conjugation in $\Gamma$ --- and thus is the identity --- and where $u_{cd}: H_0(\Gamma _c ; \mathbb{F}_p) \rightarrow H_0(\Gamma _d ; \mathbb{F}_p)$ is induced by $\Gamma _c \hookrightarrow \Gamma _d$ and $\pi _d \circ \partial '| _c$ --- and thus is identified with $$\pi _d \circ \partial ' |_c: \{\,a_cc\mid a_c \in \mathbb{F}_p \,\} \rightarrow  \{\,a_dd\mid a_d \in \mathbb{F}_p \,\}$$

Therefore, 
\begin{align*}
d^1(a_c)&= \sum _{d \in \mathcal{D}_c}  u_{cd}  (a_c) \\
&=\sum _{d \in \mathcal{D}_c}\pi _d \circ \partial '  (a_c) \\
& = \partial (a_c(\Gamma c))
\end{align*}
\end{proof}

\subsection{Proof of Theorem~\ref{t:mt}}
By the two preceding lemmas, we have for each $r \geq 2$ that the kernel of $d^r :
E^r_{k,0}
\rightarrow
E^r_{k-r,r-1}$ is infinite, which implies the infiniteness of 
$$E^\infty_{k,0} \leq
 \bigoplus _{\ell + q =k} E^\infty_{\ell,q}= H_k(\Gamma; \mathbb{F}_p)\cong H^k(\Gamma; \mathbb{F}_p)$$

\section{Existence of finite-index, residually $p$-finite subgroups of ${\bf G}(\mathcal{O}_S)$}\label{s:frpf}

In this section we give a sketch of the well-known existence statement from the title of this section. The existence essentially follows from Platonov's Theorem on finitely-generated matrix groups. We took our account below from Nica \cite{N}.

Let $w$ be a valuation of $K$ that is not contained in $S$, and let $\mathfrak{m}\se \mathcal{O}_S$ be the ideal $\{\, x \in \mathcal{O}_S \mid |x|_w<1 \,\}.$ Note that $\cap_k \mathfrak{m}^k=0$. Furthermore, $\mathcal{O}_S / \mathfrak{m}$ is identified with the values of elements of $\mathcal{O}_S$ at $w$, and hence is finite. Similarly, $\mathfrak{m}^k / \mathfrak{m}^{k+1}$ is finite for any $k \geq 1$, so that $\mathcal{O}_S / \mathfrak{m}^k$ is a finite ring.

For $k \geq 1$, let $\Lambda _k$ be the kernel of $$\alpha _k : {\bf GL_n}(\mathcal{O}_S) \rightarrow {\bf GL_n}(\mathcal{O}_S/\mathfrak{m}^k) $$ Since $\mathcal{O}_S / \mathfrak{m}^k$ is a finite ring, $\Lambda _k$ is a finite-index normal subgroup of ${\bf GL_n}(\mathcal{O}_S) $. Also note that if $m>k$ then $\Lambda _{m}$ is a normal subgroup of $\Lambda _k$ since $\Lambda _{m}$ is the kernel of $\alpha _m$ restricted to $\Lambda _k$.
 
 We claim that $\Lambda _k / \Lambda _{k+1}$ is a $p$-group. Indeed, if $g \in \Lambda _k$ then the matrix entries of $g-1$ are contained in $\mathfrak{m}^k$. Thus, the matrix entries of $(g-1)^p$ are contained  in $\mathfrak{m}^{k+1}$. Since $\mathcal{O}_S \se K$ has characteristic $p$, $g^p-1=(g-1)^p$ so that $g^p \in \Lambda _{k+1}$, establishing our claim.
 
Note that $\cap_k \mathfrak{m}^k=0$ implies $\cap_k \Lambda _k=1$. Thus, if $Z \se \Lambda_1$ is finite we can choose $k \gg 0$ such that $Z \cap \Lambda_k \se \{1\}$, and 
$$[\Lambda _1: \Lambda _k]=\prod_{i=1}^{k-1}[\Lambda _i:\Lambda _{i+1}]$$
is a power of $p$. Therefore, $\Lambda _1$ is a finite-index, residually $p$-finite subgroup of ${\bf GL_n}(\mathcal{O}_S)$.

For general ${\bf G}(\mathcal{O}_S)$ we have an embedding of $K$-groups ${\bf G} \leq {\bf GL_n}$ and we replace $\Lambda _k$ in the above with $\Lambda _k \cap {\bf G}(\mathcal{O}_S)$.


\begin{thebibliography}{BLFDT}

\bibitem{Abels} Abels, H., \textit{Finiteness properties of certain arithmetic groups in the function field case.} Israel J. Math., \textbf{76} (1991), 113-128.


\bibitem{Abramenko} Abramenko, P., \textit{Finiteness properties of Chevalley groups over $\mathbb{F}_q[t]$.} Israel J. Math., \textbf{87} (1994), 203-223.


\bibitem{Behr} Behr, H., \textit{Arithmetic groups over function fields. I. A complete characterization of finitely generated and finitely presented arithmetic subgroups of reductive algebraic groups.} J. Reine Angew. Math. \textbf{495} (1998), 79-118.

\bibitem{Brown} Brown, K., \textit{Cohomology of groups.} Springer. Graduate text in mathematics \textbf{87} (1982). 




\bibitem{B-G-W} Bux, K.-U.,  K\"{o}hl, R., and  Witzel, S., \textit{Higher finiteness properties of reductive arithmetic groups in positive characteristic: the rank theorem.} Ann. Math. \textbf{177} (2013), 311-366. 

\bibitem{B-W1} Bux, K.-U., and Wortman, K., \textit{Finiteness properties of arithmetic groups over function fields.} Invent. Math. \textbf{167} (2007), 355-378.

\bibitem{B-W2} Bux, K.-U., and Wortman, K., \textit{Connectivity properties of horospheres in Euclidean buildings and applications to finiteness properties of discrete groups.} Invent. Math. \textbf{185} (2011), 395-419.

\bibitem{D}  Dymara, J., and Osajda, D., \textit{Boundaries of right-angled hyperbolic buildings.}
  Fund. Math. \textbf{197} (2007), 123-165.





\bibitem{Ga}  Gandini, G., \textit{Bounding the homological finiteness length.} Bull. Lond. Math. Soc. \textbf{44} (2012), 1209-1214. 




\bibitem{H} Harder, G., \textit{Minkowskische	Reduktionstheorie	\"uber	Funtionenk\"orpern.} Invent. Math. \textbf{7} (1969), 33-54.


\bibitem{K} Kropholler, P., \textit{On groups of type $(FP)_\infty$.} 
J. Pure Appl. Algebra \textbf{90} (1993), 55-67.

\bibitem{N} Nica, B., \textit{Linear groups - Malcev's Theorem and Selberg's Lemma.} Preprint.

\bibitem{Sc} Schulz, B.,
\textit{Spherical subcomplexes of spherical buildings.}
Geom. Topol. \textbf{17} (2013), 531-562. 

\bibitem{St} Stuhler, U. \textit{Homological properties of certain arithmetic groups in the function field case.} Invent. Math. \textbf{57} (1980), 263-281.

\end{thebibliography}
\end{document}